\documentclass[twoside,reqno]{amsart}
\usepackage{amsfonts}
\usepackage{graphicx}
\usepackage{amscd}
\usepackage{hyperref}
\usepackage{cite}

\newtheorem{theorem}{Theorem}
\theoremstyle{plain}

\newtheorem{corollary}{Corollary}

\newtheorem{lemma}{Lemma}

\newtheorem{remark}{Remark}

\numberwithin{equation}{section}

\input{tcilatex}

\begin{document}
\title[Hierarchical fixed point problems and variational inequalities]{%
Convergence Theorems for Hierarchical Fixed Point Problems and Variational
Inequalities}
\author{\.{I}brahim Karahan}
\address{Department of Mathematics, Faculty of Science, Erzurum Technical
University, Erzurum, 25240, Turkey}
\email{ibrahimkarahan@erzurum.edu.tr}
\author{Murat \"{O}zdemir}
\address{Department of Mathematics, Faculty of Science, Ataturk University,
Erzurum, 25240, Turkey.}
\email{mozdemir@atauni.edu.tr}
\subjclass[2000]{47H10; 47J20; 47H09; 47H05}
\keywords{Variational inequality; hierarchical fixed point; convex
minimization problem; nearly nonexpansive mappings; strong convergence}

\begin{abstract}
This paper deals with a modified iterative projection method for
approximating a solution of hierarchical fixed point problems for nearly
nonexpansive mappings. Some strong convergence theorems for the proposed
method are presented under certain approximate assumptions of mappings and
parameters. As a special case, this projection method solves some quadratic
minimization problem. It should be noted that the proposed method can be
regarded as a generalized version of Wang et.al. \cite{wang}, Ceng et. al. 
\cite{ceng}, Sahu \cite{sakasa} and many other authors.
\end{abstract}

\maketitle

\section{Introduction}

Throughout this paper, $C$ is a nonempty closed convex subset of a real
Hilbert space $H$, $\left\langle \cdot ,\cdot \right\rangle $ denotes the
associated inner product, $\left\Vert \cdot \right\Vert $ stands for the
corresponding norm. Let $I$ be the idendity mapping on $C$, $P_{C}$ be the
metric projection of $H$ onto $C.$ To begin with, let us recall the
following concepts which are common use in the context of convex and
nonlinear analysis. For all $x,y\in C,$ a mapping $F:C\rightarrow H$ is said
to be monotone if $\left\langle Fx-Fy,x-y\right\rangle \geq 0$, $\eta $%
-strongly monotone if there exists a positive real number $\eta $ such that $%
\left\langle Fx-Fy,x-y\right\rangle \geq \eta \left\Vert x-y\right\Vert ^{2}$%
, $L$-Lipschitzian if there exists a positive real number $L$ such that $%
\left\Vert Fx-Fy\right\Vert \leq L\left\Vert x-y\right\Vert $. Let us also
recall that a mapping $T:C\rightarrow C$ is said to be contraction if there
exists a constant $k\in \left[ 0,1\right) $ such that $\left\Vert
Tx-Ty\right\Vert \leq k\left\Vert x-y\right\Vert $, nonexpansive if $%
\left\Vert Tx-Ty\right\Vert \leq \left\Vert x-y\right\Vert $, asymptotically
nonexpansive if for each $n\geq 1$, there exists a positive constant $%
k_{n}\geq 1$ with $\lim_{n\rightarrow \infty }k_{n}=1$\ such that $%
\left\Vert T^{n}x-T^{n}y\right\Vert \leq k_{n}\left\Vert x-y\right\Vert $,
for all $x,y\in C$.

As a generalization of asymptotically nonexpansive mappings, Sahu \cite%
{sahu1} introduced the class of nearly Lipschitzian mappings. Let fix a
sequence $\left\{ a_{n}\right\} $ in $\left[ 0,\infty \right) $\ with $%
a_{n}\rightarrow 0$. A mapping $T:C\rightarrow C$ is called nearly
Lipschitzian with respect to $\left\{ a_{n}\right\} $ if for each $n\geq 1$,
there exist a constant $k_{n}\geq 0$ such that%
\begin{equation}
\left\Vert T^{n}x-T^{n}y\right\Vert \leq k_{n}\left( \left\Vert
x-y\right\Vert +a_{n}\right) \text{, }\forall x,y\in C.  \label{aa}
\end{equation}%
The infimum constant $k_{n}$ for which (\ref{aa}) holds will be denoted by $%
\eta \left( T^{n}\right) $ and called nearly Lipschitz constant. Notice that%
\begin{equation*}
\eta \left( T^{n}\right) =\sup \left\{ \frac{\left\Vert
T^{n}x-T^{n}y\right\Vert }{\left\Vert x-y\right\Vert +a_{n}}:x,y\in C\text{, 
}x\neq y\right\} .
\end{equation*}%
A nearly Lipschitzian mapping $T$ with sequence $\left\{ a_{n},\eta \left(
T^{n}\right) \right\} $\ is said to be nearly nonexpansive if $\eta \left(
T^{n}\right) \leq 1$ for $n\geq 1$. In this paper, we study with a nearly
nonexpansive mapping which is studied by some authors (see \cite{sahu1, aos,
sakasa}). This type of mappings (not necessarily continuous) are important
with regard to be generalization of the asymptotically nonexpansive
mappings. Exactly, the class of this type of mappings is an intermediate
class between the class of asymptotically nonexpansive mappings and that of
mappings of asymptotically nonexpansive type (please see \cite{sahu1} for
nearly nonexpansive mappings examples).

Now, we focus on the hierarchical fixed point problem for a nearly
nonexpansive mapping $T$ with respect to a nonexpansive mapping $S$. This
problem is to find a point $x^{\ast }\in Fix\left( T\right) $ satisfying%
\begin{equation}
\left\langle (I-S)x^{\ast },x^{\ast }-x\right\rangle \leq 0,\text{ \ }%
\forall x\in Fix\left( T\right) ,  \label{a1}
\end{equation}%
where $Fix\left( T\right) $ is the set of fixed points of $T$, i.e., $%
Fix\left( T\right) =\left\{ x\in C:Tx=x\right\} $. It is easy to see that
the problem (\ref{a1}) is equivalent to the problem of finding a point $%
x^{\ast }\in C$ that satisfies $x^{\ast }=P_{Fix\left( T\right) }Sx^{\ast }$.

Let $N_{Fix\left( T\right) }$ be the normal cone to $Fix\left( T\right) $
defined by%
\begin{equation*}
N_{Fix\left( T\right) }x=\left\{ 
\begin{array}{cc}
\left\{ u\in H:\left\langle y-x,u\right\rangle \leq 0,\text{ }\forall y\in
Fix\left( T\right) \right\} , & \text{if }x\in Fix\left( T\right) \\ 
\emptyset , & \text{ \ if }x\notin Fix\left( T\right) .%
\end{array}%
\right.
\end{equation*}%
Then, the hierarchical fixed point problem is equivalent to the variational
inclusion problem which is to find a point $x^{\ast }\in C$ such that%
\begin{equation*}
0\in \left( I-S\right) x^{\ast }+N_{Fix\left( T\right) }x^{\ast }.
\end{equation*}%
The existence problem of hierarchical fixed points for a nonlinear mapping
and approximation problem has been studied by several authors (see \cite%
{yao, moudafi, abdul, cian, marxu, xu, deng, cenpe}). In 2006, Marino and Xu 
\cite{marino} introduced the following viscosity iterative method:%
\begin{equation}
x_{n+1}=\alpha _{n}\gamma f(x_{n})+\left( I-\alpha _{n}A\right) Tx_{n},\text{
}\forall n\geq 0,  \label{f1}
\end{equation}%
where $f$ is a contraction, $T$ is a nonexpansive mapping and $A$ is a
strongly positive bounded linear operator on $H$; i.e., $\left\langle
Ax,x\right\rangle \geq \gamma \left\Vert x\right\Vert ,$ $\forall x\in H$
for some $\gamma >0$. Under the appropriate conditions, they proved that the
sequence $\{x_{n}\}$ defined by (\ref{f1}) converges strongly to the unique
solution of the variational inequality 
\begin{equation}
\left\langle \left( \gamma f-A\right) x^{\ast },x-x^{\ast }\right\rangle
\leq 0,\text{ }\forall x\in C,  \label{v1}
\end{equation}%
which is the optimality condition for the minimization problem%
\begin{equation*}
\min_{x\in C}\frac{1}{2}\left\langle Ax,x\right\rangle -h(x)
\end{equation*}%
where $h$ is a potential function for $\gamma f$ i.e., $h^{\prime
}(x)=\gamma f(x)$ for all $x\in H$.

In 2011, Ceng et al. \cite{ceng} generalized the iterative method of Marino
and Xu \cite{marino} by taking a Lipschitzian mapping $V$ and Lipschitzian
and strongly monotone operator $F$ instead of the mappings $f$ and $A$,
respectively. They gave the following iterative method:%
\begin{equation}
x_{n+1}=P_{C}\left[ \alpha _{n}\rho Vx_{n}+\left( I-\alpha _{n}\mu F\right)
Tx_{n}\right] ,\text{ }\forall n\geq 0,  \label{f2}
\end{equation}%
where $P_{C}$ is a metric projection and $T$ is a nonexpansive mapping and
also proved that the sequence $\left\{ x_{n}\right\} $ generated by (\ref{f2}%
) converges strongly to the unique solution of the variational inequality%
\begin{equation}
\left\langle \left( \rho V-\mu F\right) x^{\ast },x-x^{\ast }\right\rangle
\leq 0,\text{ }\forall x\in Fix(T).  \label{v2}
\end{equation}%
Recently, motivated by the iteration method (\ref{f2}) of Ceng et al. \cite%
{ceng}, Wang and Xu \cite{wang} studied on the following iterative method
for a hierarchical fixed point problem:%
\begin{equation}
\left\{ 
\begin{array}{l}
y_{n}=\beta _{n}Sx_{n}+\left( 1-\beta _{n}\right) x_{n},\text{ \ \ \ \ \ \ \
\ \ \ \ \ \ \ \ \ \ \ \ \ \ \ \ } \\ 
x_{n+1}=P_{C}\left[ \alpha _{n}\rho Vx_{n}+\left( I-\alpha _{n}\mu F\right)
Ty_{n}\right] ,\text{ }\forall n\geq 0,%
\end{array}%
\right.  \label{f4}
\end{equation}%
where $S,T:C\rightarrow C$ are nonexpansive mappings, $V:C\rightarrow H$ is
a $\gamma $-Lipschitzian mapping and $F:C\rightarrow H$ is a $L$%
-Lipschitzian and $\eta $-strongly monotone operator. They proved that under
some suitable assumptions on the sequences $\left\{ \alpha _{n}\right\} $
and $\left\{ \beta _{n}\right\} $, the sequence $\{x_{n}\}$ generated by (%
\ref{f4}) converges strongly to the hierarchical fixed point of $T$ with
respect to the mapping $S$ which is the unique solution of the variational
inequality (\ref{v2}). With this study, Wang and Xu extends and improves the
many recent results of other authors.

Let $\left\{ T_{n}\right\} $ be a sequence of mappings from $C$ into $H$ and
fix a sequence $\left\{ a_{n}\right\} $ in $\left[ 0,\infty \right) $\ with $%
a_{n}\rightarrow 0$. Then, $\left\{ T_{n}\right\} $ is called a sequence of
nearly nonexpansive mappings \cite{wong} with respect to a sequence $\left\{
a_{n}\right\} $ if%
\begin{equation}
\left\Vert T_{n}x-T_{n}y\right\Vert \leq \left\Vert x-y\right\Vert +a_{n}%
\text{, }\forall x,y\in C\text{, }n\geq 1.  \label{nearly}
\end{equation}%
It is obvious that the sequence of nearly nonexpansive mappings is a wider
class of sequence of nonexpansive mappings. For the sequence of nearly
nonexpansive mappings defined by (\ref{nearly}), Sahu et. al. \cite{sahu2}
introduced a new iteration method to solve the hierarchical fixed point
problem and variational inequality problem.

\begin{remark}
\label{AA} Let $\left\{ T_{n}\right\} $ be a sequence of nearly nonexpansive
mappings. Then,

\begin{enumerate}
\item for each $n\geq 1$, $T_{n}$\ is not a nearly nonexpansive mapping.

\item if $T$ is a mapping on $C$ defined by $Tx=\lim_{n\rightarrow \infty
}T_{n}x$ for all $x\in C$, then, it is clear that $T$ is a nonexpansive
mapping.
\end{enumerate}
\end{remark}

Recently, in 2012, Sahu et al. \cite{sakasa} introduced the following
iterative method for the sequence of nearly nonexpansive mappings $\left\{
T_{n}\right\} $ defined by (\ref{nearly})%
\begin{equation}
x_{n+1}=P_{C}\left[ \alpha _{n}\rho Vx_{n}+\left( 1-\alpha _{n}\mu F\right)
T_{n}x_{n}\right] ,\text{ }\forall n\geq 1.  \label{f3}
\end{equation}%
They proved that the sequence $\{x_{n}\}$ generated by (\ref{f3}) converges
strongly to the unique solution of the variational inequality (\ref{v2}).

\begin{remark}
\label{AB} Since the mapping $T_{n}$ is not a nearly nonexpansive mapping
for each $n\geq 1$, if one take $T_{n}:=T$ for all $n\geq 1$ such that $T$
is a nearly nonexpansive mapping, then the iteration (\ref{f3}) is not well
defined. Hence, the main result of Sahu et. al. \cite{sakasa} is no longer
valid for a nearly nonexpansive mapping.
\end{remark}

In this paper, motivated and inspired by the work of Wang and Xu \cite{wang}%
, we introduce a modified iterative projection method to find a hierarchical
fixed point of a nearly nonexpansive mapping with respect to a nonexpansive
mapping. We show that our iterative method converges strongly to the unique
solution of the variational inequality (\ref{v2}). As a special case,
presented projection method solves some quadratic minimization problem.
Also, our method improves and generalizes corresponding results of Yao et.
al. \cite{yao}, Marino and Xu \cite{marino}, Ceng et. al. \cite{ceng}, Wang
and Xu \cite{wang}, Moudafi \cite{mou}, Xu \cite{xu1}, Tian \cite{tian} and
Suzuki \cite{suzuki}.

\section{Preliminaries}

This section contains some lemmas and definitions which will be used in the
proof of our main result in the following section. We write $%
x_{n}\rightharpoonup x$ to indicate that the sequence $\left\{ x_{n}\right\} 
$ converges weakly to $x,$ and $x_{n}\rightarrow x$ for the strong
convergence. A mapping $P_{C}:H\rightarrow C$ is called a metric projection
if there exists a unique nearest point in $C$ denoted by $P_{C}x$ such that%
\begin{equation*}
\left\Vert x-P_{C}x\right\Vert =\inf_{y\in C}\left\Vert x-y\right\Vert \text{%
, }\forall x\in H
\end{equation*}%
It is easy to see that $P_{C}$ is a nonexpansive mapping and it satisfies
the following inequality:%
\begin{equation}
\left\langle x-P_{C}x,y-P_{C}x\right\rangle \leq 0,\text{ }\forall x\in
H,y\in C.  \label{c1}
\end{equation}%
Now, we give the definitions of a demicontinuous mapping, asymptotic radius
and asmptotic center.

Let $C$ be a nonempty subset of a Banach space $X$ and $T:C\rightarrow C$ be
a mapping. $T$ is called demicontinuous if, whenever a sequence $\left\{
x_{n}\right\} $ in $C$ converges strongly to $x\in X$, then $\left\{
Tx_{n}\right\} $ converges weakly to $Tx.$

Let $C$ be a nonempty closed convex subset of a uniformly convex Banach
space $X$, $\left\{ x_{n}\right\} $ be a bounded sequence in $X$ and $%
r:C\rightarrow \left[ 0,\infty \right) $ be a functional defined by%
\begin{equation*}
r\left( x\right) =\limsup_{n\rightarrow \infty }\left\Vert
x_{n}-x\right\Vert ,\text{ }x\in C.
\end{equation*}%
The infimum of $r\left( \cdot \right) $ over $C$ is called asymptotic radius
of $\left\{ x_{n}\right\} $ with respect to $C$ and is denoted by $r\left(
C,\left\{ x_{n}\right\} \right) $. A point $\widetilde{x}\in C$ is said to
be an asymptotic center of the sequence $\left\{ x_{n}\right\} $ with
respect to $C$ if%
\begin{equation*}
r\left( \widetilde{x}\right) =\min \left\{ r\left( x\right) :x\in C\right\} .
\end{equation*}%
The set of all asymptotic centers is denoted by $A\left( C,\left\{
x_{n}\right\} \right) $. Related with these definitions, we will use the
followings in our main results.

\begin{theorem}
\label{C}\cite{AOS} Let $C$ be a nonempty closed convex subset of a
uniformly convex Banach space $X$ satisfying the Opial condition. If $%
\left\{ x_{n}\right\} $ is a sequence in $C$ such that $x_{n}\rightharpoonup 
$ $\widetilde{x}$, then $\widetilde{x}$ is the asymptotic center of $\left\{
x_{n}\right\} $ in $C$.
\end{theorem}

\begin{lemma}
\label{b}\cite{sahu1} Let $C$ be a nonempty closed convex subset of a
uniformly convex Banach space $X$ and $T:C\rightarrow C$ be a demicontinuous
nearly Lipschitzian mapping with sequence $\left\{ a_{n},\eta \left(
T^{n}\right) \right\} $ such that $\lim_{n\rightarrow \infty }\eta \left(
T^{n}\right) \leq 1$. If $\left\{ x_{n}\right\} $ is a bounded sequence in $%
C $ such that%
\begin{equation*}
\lim_{m\rightarrow \infty }\left( \lim_{n\rightarrow \infty }\left\Vert
x_{n}-T^{m}x_{n}\right\Vert \right) =0\text{ and }A\left( C,\left\{
x_{n}\right\} \right) =\left\{ \widetilde{x}\right\} ,
\end{equation*}%
then $\widetilde{x}$ is a fixed point of $T$.
\end{lemma}

\begin{lemma}
\label{B}\cite{ceng} Let $V:C\rightarrow H$ be a $\gamma $-Lipschitzian
mapping and let $F:C\rightarrow H$ be a $L$-Lipschitzian and $\eta $%
-strongly monotone operator, then for $0\leq \rho \gamma <\mu \eta ,$%
\begin{equation*}
\left\langle \left( \mu F-\rho V\right) x-\left( \mu F-\rho V\right)
y,x-y\right\rangle \geq \left( \mu \eta -\rho \gamma \right) \left\Vert
x-y\right\Vert ^{2},\text{ }\forall x,y\in C.
\end{equation*}%
That is to say, the operator $\mu F-\rho V$ is $\mu \eta -\rho \gamma $%
-strongly monotone.
\end{lemma}

\begin{lemma}
\label{c}\cite{yamada} Let $C$ be a nonempty subset of a real Hilbert space $%
H.$ Suppose that $\lambda \in \left( 0,1\right) $ and $\mu >0$. Let $%
F:C\rightarrow H$ be a $L$-Lipschitzian and $\eta $-strongly monotone
operator. Define the mapping $G:C\rightarrow H$ by%
\begin{equation*}
Gx=x-\lambda \mu Fx\text{, }\forall x\in C.
\end{equation*}%
Then, $G$ is a contraction that provided $\mu <2\eta /L^{2}$. More
precisely, for $\mu \in \left( 0,2\eta /L^{2}\right) $,%
\begin{equation*}
\left\Vert Gx-Gy\right\Vert \leq \left( 1-\lambda \nu \right) \left\Vert
x-y\right\Vert \text{, }\forall x,y\in C,
\end{equation*}%
where $\nu =1-\sqrt{1-\mu \left( 2\eta -\mu L^{2}\right) .}$
\end{lemma}

\begin{lemma}
\label{Y}\cite{xu} Assume that $\left\{ x_{n}\right\} $ is a sequence of
nonnegative real numbers such that%
\begin{equation*}
x_{n+1}\leq \left( 1-\alpha _{n}\right) x_{n}+\alpha _{n}\beta _{n},\text{ }%
\forall n\geq 0
\end{equation*}%
where $\left\{ \alpha _{n}\right\} $ \ and $\left\{ \beta _{n}\right\} $ are
sequences of real numbers which satisfy the following conditions: 
\begin{eqnarray*}
&\text{(i)}&\left\{ \alpha _{n}\right\} \subset \left[ 0,1\right] \text{ and 
}\tsum_{n=1}^{\infty }\alpha _{n}=\infty \text{,\ \ \ \ \ \ \ \ \ \ \ \ \ \
\ \ \ \ \ \ \ \ \ \ \ \ \ \ \ \ \ \ \ \ \ \ \ \ \ \ \ \ \ \ \ \ \ \ \ \ \ \
\ } \\
&\text{(ii)}&\text{ either }\limsup_{n\rightarrow \infty }\beta _{n}\leq 0%
\text{ or }\sum_{n=1}^{\infty }\alpha _{n}\beta _{n}<\infty .
\end{eqnarray*}%
Then $\lim_{n\rightarrow \infty }x_{n}=0.$
\end{lemma}

\section{Main result}

\begin{theorem}
\label{X} Let $C$ be a nonempty closed convex subset of a real Hilbert space 
$H.$ Let $S$ be a nonexpansive mapping and $T$ be a demicontinuous nearly
nonexpansive mapping on $C$ with respect to the sequence $\left\{
a_{n}\right\} $ such that $Fix\left( T\right) \neq \emptyset $. Let $%
V:C\rightarrow H$ be a $\gamma $-Lipschitzian mapping, $F:C\rightarrow H$ be
a $L$-Lipschitzian and $\eta $-strongly monotone operator such that the
coefficients satisfy $0<\mu <\frac{2\eta }{L^{2}}$, $0\leq \rho \gamma <\nu $%
, where $\nu =1-\sqrt{1-\mu \left( 2\eta -\mu L^{2}\right) }$. For an
arbitrarily initial value $x_{1}\in C,$ consider the sequence $\left\{
x_{n}\right\} $ in $C$ generated by%
\begin{equation}
\left\{ 
\begin{array}{l}
y_{n}=\beta _{n}Sx_{n}+\left( 1-\beta _{n}\right) x_{n} \\ 
x_{n+1}=P_{C}\left[ \alpha _{n}\rho Vx_{n}+\left( I-\alpha _{n}\mu F\right)
T^{n}y_{n}\right] ,\text{ }\forall n\geq 1,%
\end{array}%
\right.   \label{4}
\end{equation}%
where $\left\{ \alpha _{n}\right\} $ and $\left\{ \beta _{n}\right\} $ are
sequences in $\left[ 0,1\right] $ satisfying the conditions:%
\begin{eqnarray*}
&(i)&\lim_{n\rightarrow \infty }\alpha _{n}=0\text{, and }%
\tsum_{n=1}^{\infty }\alpha _{n}=\infty \text{;} \\
&(ii)&\lim_{n\rightarrow \infty }\frac{a_{n}}{\alpha _{n}}=0\text{, }%
\lim_{n\rightarrow \infty }\frac{\beta _{n}}{\alpha _{n}}=0\text{, }%
\lim_{n\rightarrow \infty }\frac{\left\vert \alpha _{n}-\alpha
_{n-1}\right\vert }{\alpha _{n}}=0\text{ and } \\
&&\lim_{n\rightarrow \infty }\frac{\left\vert \beta _{n}-\beta
_{n-1}\right\vert }{\alpha _{n}}=0\text{;} \\
&(iii)&\lim_{n\rightarrow \infty }\left\Vert T^{n}x-T^{n-1}x\right\Vert =0%
\text{ and }\lim_{n\rightarrow \infty }\frac{\left\Vert
T^{n}x-T^{n-1}x\right\Vert }{\alpha _{n}}=0,\forall x\in C\text{. \ \ \ \ \
\ \ }
\end{eqnarray*}%
Then, the sequence $\left\{ x_{n}\right\} $ converges strongly to $x^{\ast
}\in Fix\left( T\right) $, where $x^{\ast }$\ is the unique solution of the
variational inequality (\ref{v2}).
\end{theorem}

\begin{proof}
Since the mapping $\mu F-\rho V$ is a strongly monotone operator from Lemma %
\ref{B},\ it is known that the variational inequality (\ref{v2}) has an
unique solution. Let denote this solution by $x^{\ast }\in Fix\left(
T\right) $. Now, we divide our proof into five steps.

\textbf{Step 1. }First we show that the sequence $\left\{ x_{n}\right\} $
generated by (\ref{4}) is bounded. From condition (ii), without loss of
generality, we may suppose that $\beta _{n}\leq \alpha _{n}$, for all $n\geq
1$. Hence, we get $\lim_{n\rightarrow \infty }\beta _{n}=0$.\ Let $p\in
Fix\left( T\right) $ and $t_{n}=\alpha _{n}\rho Vx_{n}+\left( I-\alpha
_{n}\mu F\right) T^{n}y_{n}$.\ Then we have%
\begin{eqnarray}
\left\Vert y_{n}-p\right\Vert &=&\left\Vert \beta _{n}Sx_{n}+\left( 1-\beta
_{n}\right) x_{n}-p\right\Vert  \notag \\
&\leq &\left( 1-\beta _{n}\right) \left\Vert x_{n}-p\right\Vert +\beta
_{n}\left\Vert Sx_{n}-p\right\Vert  \notag \\
&\leq &\left( 1-\beta _{n}\right) \left\Vert x_{n}-p\right\Vert +\beta
_{n}\left\Vert Sx_{n}-Sp\right\Vert +\beta _{n}\left\Vert Sp-p\right\Vert 
\notag \\
&\leq &\left\Vert x_{n}-p\right\Vert +\beta _{n}\left\Vert Sp-p\right\Vert ,
\label{2}
\end{eqnarray}%
and by using the definition of nearly nonexpansive mapping and Lemma \ref{c}%
, we obtain%
\begin{eqnarray}
\left\Vert x_{n+1}-p\right\Vert &=&\left\Vert P_{C}t_{n}-P_{C}p\right\Vert 
\notag \\
&\leq &\left\Vert t_{n}-p\right\Vert  \notag \\
&=&\left\Vert \alpha _{n}\rho Vx_{n}+\left( I-\alpha _{n}\mu F\right)
T^{n}y_{n}-p\right\Vert  \notag \\
&\leq &\alpha _{n}\left\Vert \rho Vx_{n}-\mu Fp\right\Vert +\left\Vert
\left( I-\alpha _{n}\mu F\right) T^{n}y_{n}-\left( I-\alpha _{n}\mu F\right)
T^{n}p\right\Vert  \notag \\
&\leq &\alpha _{n}\rho \gamma \left\Vert x_{n}-p\right\Vert +\alpha
_{n}\left\Vert \rho Vp-\mu Fp\right\Vert  \notag \\
&&+\left( 1-\alpha _{n}\nu \right) \left( \left\Vert y_{n}-p\right\Vert
+a_{n}\right) .  \label{3}
\end{eqnarray}%
From (\ref{2}) and (\ref{3}), we get%
\begin{eqnarray*}
\left\Vert x_{n+1}-p\right\Vert &\leq &\alpha _{n}\rho \gamma \left\Vert
x_{n}-p\right\Vert +\alpha _{n}\left\Vert \rho Vp-\mu Fp\right\Vert \\
&&+\left( 1-\alpha _{n}\nu \right) \left( \left\Vert x_{n}-p\right\Vert
+\beta _{n}\left\Vert Sp-p\right\Vert +a_{n}\right) \\
&\leq &\left( 1-\alpha _{n}\left( \nu -\rho \gamma \right) \right)
\left\Vert x_{n}-p\right\Vert \\
&&+\alpha _{n}\left( \nu -\rho \gamma \right) \left[ \frac{1}{\left( \nu
-\rho \gamma \right) }\left( \left\Vert \rho Vp-\mu Fp\right\Vert
+\left\Vert Sp-p\right\Vert +\frac{a_{n}}{\alpha _{n}}\right) \right] .
\end{eqnarray*}%
From condition (ii), this sequence is bounded, and so we can write%
\begin{equation*}
\left\Vert x_{n+1}-p\right\Vert \leq \left( 1-\alpha _{n}\left( \nu -\rho
\gamma \right) \right) \left\Vert x_{n}-p\right\Vert +\alpha _{n}\left( \nu
-\rho \gamma \right) M,
\end{equation*}%
where%
\begin{equation*}
\frac{1}{\left( \nu -\rho \gamma \right) }\left( \left\Vert \rho Vp-\mu
Fp\right\Vert +\left\Vert Sp-p\right\Vert +\frac{a_{n}}{\alpha _{n}}\right)
\leq M\text{, }\forall n\geq 1
\end{equation*}%
Therefore, by induction, we get%
\begin{equation*}
\left\Vert x_{n+1}-p\right\Vert \leq \max \left\{ \left\Vert
x_{1}-p\right\Vert \text{, }M\right\} .
\end{equation*}%
Hence, we obtain that $\left\{ x_{n}\right\} $ is bounded. So, the sequences 
$\left\{ y_{n}\right\} $, $\left\{ Tx_{n}\right\} $, $\left\{ Sx_{n}\right\} 
$, $\left\{ Vx_{n}\right\} $ and $\left\{ FTy_{n}\right\} $ are bounded too.

\textbf{Step 2.} Secondly, we show that $\lim_{n\rightarrow \infty
}\left\Vert x_{n+1}-x_{n}\right\Vert =0$. By using the iteration (\ref{4}),
we have%
\begin{eqnarray}
\left\Vert y_{n}-y_{n-1}\right\Vert &=&\left\Vert \beta _{n}Sx_{n}+\left(
1-\beta _{n}\right) x_{n}-\beta _{n-1}Sx_{n-1}-\left( 1-\beta _{n-1}\right)
x_{n-1}\right\Vert  \notag \\
&\leq &\beta _{n}\left\Vert Sx_{n}-Sx_{n-1}\right\Vert +\left( 1-\beta
_{n}\right) \left\Vert x_{n}-x_{n-1}\right\Vert  \notag \\
&&+\left\vert \beta _{n}-\beta _{n-1}\right\vert \left( \left\Vert
Sx_{n-1}\right\Vert +\left\Vert x_{n-1}\right\Vert \right)  \notag \\
&\leq &\left\Vert x_{n}-x_{n-1}\right\Vert +\left\vert \beta _{n}-\beta
_{n-1}\right\vert M_{1},  \label{b1}
\end{eqnarray}%
where $M_{1}$ is a constant such that $\sup_{n\geq 1}\left\{ \left\Vert
Sx_{n}\right\Vert +\left\Vert x_{n}\right\Vert \right\} \leq M_{1},$ and%
\begin{eqnarray}
\left\Vert x_{n+1}-x_{n}\right\Vert &\leq &\left\Vert
P_{C}t_{n}-P_{C}t_{n-1}\right\Vert  \notag \\
&\leq &\left\Vert \alpha _{n}\rho Vx_{n}+\left( I-\alpha _{n}\mu F\right)
T^{n}y_{n}\right.  \notag \\
&&\left. -\alpha _{n-1}\rho Vx_{n-1}-\left( I-\alpha _{n-1}\mu F\right)
T^{n-1}y_{n-1}\right\Vert  \notag \\
&\leq &\left\Vert \alpha _{n}\rho V\left( x_{n}-x_{n-1}\right) +\left(
\alpha _{n}-\alpha _{n-1}\right) \rho Vx_{n-1}\right.  \notag \\
&&+\left( I-\alpha _{n}\mu F\right) T^{n}y_{n}-\left( I-\alpha _{n}\mu
F\right) T^{n}y_{n-1}  \notag \\
&&\left. +T^{n}y_{n-1}-T^{n-1}y_{n-1}+\alpha _{n-1}\mu
FT^{n-1}y_{n-1}-\alpha _{n}\mu FT^{n}y_{n-1}\right\Vert  \notag \\
&\leq &\alpha _{n}\rho \gamma \left\Vert x_{n}-x_{n-1}\right\Vert +\rho
\left\vert \alpha _{n}-\alpha _{n-1}\right\vert \left\Vert
Vx_{n-1}\right\Vert  \notag \\
&&+\left( 1-\alpha _{n}\nu \right) \left\Vert
T^{n}y_{n}-T^{n}y_{n-1}\right\Vert +\left\Vert
T^{n}y_{n-1}-T^{n-1}y_{n-1}\right\Vert  \notag \\
&&+\mu \left\Vert \alpha _{n-1}FT^{n-1}y_{n-1}-\alpha
_{n}FT^{n}y_{n-1}\right\Vert  \notag \\
&\leq &\alpha _{n}\rho \gamma \left\Vert x_{n}-x_{n-1}\right\Vert +\rho
\left\vert \alpha _{n}-\alpha _{n-1}\right\vert \left\Vert
Vx_{n-1}\right\Vert  \notag \\
&&+\left( 1-\alpha _{n}\nu \right) \left[ \left\Vert
y_{n}-y_{n-1}\right\Vert +a_{n}\right] +\left\Vert
T^{n}y_{n-1}-T^{n-1}y_{n-1}\right\Vert  \notag \\
&&+\mu \left\Vert \alpha _{n-1}\left( FT^{n-1}y_{n-1}-FT^{n}y_{n-1}\right)
\right.  \notag \\
&&\left. -\left( \alpha _{n}-\alpha _{n-1}\right) FT^{n}y_{n-1}\right\Vert .
\label{b2}
\end{eqnarray}%
So, from (\ref{b1}) and (\ref{b2}), we get%
\begin{eqnarray*}
\left\Vert x_{n+1}-x_{n}\right\Vert &\leq &\alpha _{n}\rho \gamma \left\Vert
x_{n}-x_{n-1}\right\Vert +\rho \left\vert \alpha _{n}-\alpha
_{n-1}\right\vert \left\Vert Vx_{n-1}\right\Vert \\
&&+\left( 1-\alpha _{n}\nu \right) \left\Vert x_{n}-x_{n-1}\right\Vert
+\left( 1-\alpha _{n}\nu \right) \left\vert \beta _{n}-\beta
_{n-1}\right\vert M_{1} \\
&&+\left( 1-\alpha _{n}\nu \right) a_{n}+\left\Vert
T^{n}y_{n-1}-T^{n-1}y_{n-1}\right\Vert \\
&&+\mu \alpha _{n-1}L\left\Vert T^{n}y_{n-1}-T^{n-1}y_{n-1}\right\Vert
+\left\vert \alpha _{n}-\alpha _{n-1}\right\vert \left\Vert
FT^{n}y_{n-1}\right\Vert \\
&\leq &\left( 1-\alpha _{n}\left( \nu -\rho \gamma \right) \right)
\left\Vert x_{n}-x_{n-1}\right\Vert \\
&&+\left\vert \alpha _{n}-\alpha _{n-1}\right\vert \left( \rho \left\Vert
Vx_{n-1}\right\Vert +\left\Vert FT^{n}y_{n-1}\right\Vert \right) \\
&&+\left( 1+\mu \alpha _{n-1}L\right) \left\Vert
T^{n}y_{n-1}-T^{n-1}y_{n-1}\right\Vert +\left\vert \beta _{n}-\beta
_{n-1}\right\vert M_{1}+a_{n} \\
&\leq &\left( 1-\alpha _{n}\left( \nu -\rho \gamma \right) \right)
\left\Vert x_{n}-x_{n-1}\right\Vert +\alpha _{n}\left( \nu -\rho \gamma
\right) \delta _{n},
\end{eqnarray*}%
where%
\begin{equation*}
\delta _{n}=\frac{1}{\left( \nu -\rho \gamma \right) }\left[ 
\begin{array}{c}
\left( 1+\mu \alpha _{n-1}L\right) \frac{\left\Vert
T^{n}y_{n-1}-T^{n-1}y_{n-1}\right\Vert }{\alpha _{n}} \\ 
+\left( \left\vert \frac{\alpha _{n}-\alpha _{n-1}}{\alpha _{n}}\right\vert
+\left\vert \frac{\beta _{n}-\beta _{n-1}}{\alpha _{n}}\right\vert \right)
M_{2}+\frac{a_{n}}{\alpha _{n}}%
\end{array}%
\right] ,
\end{equation*}%
and%
\begin{equation*}
\sup_{n\geq 1}\left\{ \rho \left\Vert Vx_{n-1}\right\Vert +\left\Vert
FT_{n}y_{n-1}\right\Vert ,\text{ }M_{1}\right\} \leq M_{2}.
\end{equation*}%
By using the conditions (ii) and (iii), since $\limsup_{n\rightarrow \infty
}\delta _{n}\leq 0$, it follows from Lemma \ref{Y} that%
\begin{equation}
\left\Vert x_{n+1}-x_{n}\right\Vert \rightarrow 0\text{, as }n\rightarrow
\infty .  \label{1}
\end{equation}

\textbf{Step 3.} Next, we show that $\lim_{n\rightarrow \infty }\left\Vert
x_{n}-Tx_{n}\right\Vert =0$. For $n\geq m\geq 1$, we get%
\begin{eqnarray}
\left\Vert T^{n}y_{n}-T^{m}x_{n}\right\Vert  &\leq &\left\Vert
T^{n}y_{n}-T^{n-1}y_{n}\right\Vert +\left\Vert
T^{n-1}y_{n}-T^{n-2}y_{n}\right\Vert   \notag \\
&&+\cdots +\left\Vert T^{m}y_{n}-T^{m}x_{n}\right\Vert   \notag \\
&\leq &\left\Vert T^{n}y_{n}-T^{n-1}y_{n}\right\Vert +\left\Vert
T^{n-1}y_{n}-T^{n-2}y_{n}\right\Vert   \notag \\
&&+\cdots +\left\Vert y_{n}-x_{n}\right\Vert +a_{m}\text{,}  \label{d1}
\end{eqnarray}%
and so%
\begin{eqnarray}
\left\Vert x_{n+1}-T^{m}x_{n}\right\Vert  &=&\left\Vert
P_{C}t_{n}-P_{C}T^{m}x_{n}\right\Vert   \notag \\
&\leq &\left\Vert \alpha _{n}\rho Vx_{n}+\left( I-\alpha _{n}\mu F\right)
T^{n}y_{n}-T^{m}x_{n}\right\Vert   \notag \\
&\leq &\alpha _{n}\left\Vert \rho Vx_{n}-\mu FT^{n}y_{n}\right\Vert
+\left\Vert T^{n}y_{n}-T^{m}x_{n}\right\Vert   \notag \\
&\leq &\alpha _{n}\left\Vert \rho Vx_{n}-\mu FT^{n}y_{n}\right\Vert
+\left\Vert T^{n}y_{n}-T^{n-1}y_{n}\right\Vert   \notag \\
&&+\left\Vert T^{n-1}y_{n}-T^{n-2}y_{n}\right\Vert +\cdots +\left\Vert
y_{n}-x_{n}\right\Vert +a_{m}\text{.}  \label{d2}
\end{eqnarray}%
Hence, we obtain from (\ref{d1}) and (\ref{d2}) 
\begin{eqnarray}
\left\Vert x_{n}-T^{m}x_{n}\right\Vert  &\leq &\left\Vert
x_{n}-x_{n+1}\right\Vert +\left\Vert x_{n+1}-T^{m}x_{n}\right\Vert   \notag
\\
&\leq &\left\Vert x_{n}-x_{n+1}\right\Vert +\alpha _{n}\left\Vert \rho
Vx_{n}-\mu FT^{n}y_{n}\right\Vert   \notag \\
&&+\left\Vert T^{n}y_{n}-T^{n-1}y_{n}\right\Vert +\left\Vert
T^{n-1}y_{n}-T^{n-2}y_{n}\right\Vert   \notag \\
&&+\cdots +\left\Vert y_{n}-x_{n}\right\Vert +a_{m}  \notag \\
&\leq &\left\Vert x_{n}-x_{n+1}\right\Vert +\alpha _{n}\left\Vert \rho
Vx_{n}-\mu FT^{n}y_{n}\right\Vert   \notag \\
&&+\left\Vert T^{n}y_{n}-T^{n-1}y_{n}\right\Vert +\left\Vert
T^{n-1}y_{n}-T^{n-2}y_{n}\right\Vert   \notag \\
&&+\cdots +\beta _{n}\left\Vert Sx_{n}-x_{n}\right\Vert +a_{m}  \label{c2}
\end{eqnarray}%
Since $\left\Vert \rho Vx_{n}-\mu FT^{n}y_{n}\right\Vert $ and $\left\Vert
Sx_{n}-x_{n}\right\Vert $ are bounded,\ it follows from (\ref{1}\textbf{),} (%
\ref{c2}\textbf{), }condition (i)\textbf{\ }and condition (iii) that 
\begin{equation}
\lim_{m\rightarrow \infty }\left( \lim_{n\rightarrow \infty }\left\Vert
x_{n}-T^{m}x_{n}\right\Vert \right) =0\text{.}  \label{6}
\end{equation}%
Combining (\ref{6}\textbf{) }and condition (iii), we have 
\begin{equation*}
\left\Vert x_{n}-Tx_{n}\right\Vert \leq \left\Vert
x_{n}-T^{m}x_{n}\right\Vert +\left\Vert T^{m}x_{n}-Tx_{n}\right\Vert
\rightarrow 0\text{, as }n,m\rightarrow \infty .
\end{equation*}

\textbf{Step 4.} Now, we show that $\limsup_{n\rightarrow \infty
}\left\langle \left( \rho V-\mu F\right) x^{\ast },x_{n}-x^{\ast
}\right\rangle \leq 0$, where $x^{\ast }$\ is the unique solution of the
variational inequality (\ref{v2}). Since the sequence $\left\{ x_{n}\right\} 
$ is bounded, it has a weak convergent subsequence $\left\{
x_{n_{k}}\right\} $ such that%
\begin{equation*}
\limsup_{n\rightarrow \infty }\left\langle \left( \rho V-\mu F\right)
x^{\ast },x_{n}-x^{\ast }\right\rangle =\limsup_{k\rightarrow \infty
}\left\langle \left( \rho V-\mu F\right) x^{\ast },x_{n_{k}}-x^{\ast
}\right\rangle .
\end{equation*}%
Let $x_{n_{k}}\rightharpoonup \widetilde{x}$, as\textbf{\ }$k\rightarrow
\infty $\textbf{. }Then, Opial's condition guarantee that the weakly
subsequential limit of $\left\{ x_{n}\right\} $ is unique. Hence, this
implies that $x_{n}\rightharpoonup \widetilde{x}$, as\textbf{\ }$%
n\rightarrow \infty $.\ So, it follows from (\ref{6}\textbf{), }Theorem \ref%
{C} and Lemma \ref{b} that\textbf{\ }$\widetilde{x}\in Fix\left( T\right) $%
\textbf{.} Therefore%
\begin{equation*}
\limsup_{n\rightarrow \infty }\left\langle \left( \rho V-\mu F\right)
x^{\ast },x_{n}-x^{\ast }\right\rangle =\left\langle \left( \rho V-\mu
F\right) x^{\ast },\widetilde{x}-x^{\ast }\right\rangle \leq 0.
\end{equation*}

\textbf{Step 5:} Finally, we show that the sequence $\left\{ x_{n}\right\} $
converges strongly to $x^{\ast }$. By using the inequality (\ref{c1}\textbf{)%
}, we have%
\begin{eqnarray*}
\left\Vert x_{n+1}-x^{\ast }\right\Vert ^{2} &=&\left\langle
P_{C}t_{n}-x^{\ast },x_{n+1}-x^{\ast }\right\rangle  \\
&=&\left\langle P_{C}t_{n}-t_{n},x_{n+1}-x^{\ast }\right\rangle
+\left\langle t_{n}-x^{\ast },x_{n+1}-x^{\ast }\right\rangle  \\
&\leq &\left\langle \alpha _{n}\rho Vx_{n}+\left( I-\alpha _{n}\mu F\right)
T^{n}y_{n}-x^{\ast },x_{n+1}-x^{\ast }\right\rangle  \\
&=&\left\langle \alpha _{n}\left( \rho Vx_{n}-\mu Fx^{\ast }\right) +\left(
I-\alpha _{n}\mu F\right) T^{n}y_{n}\right.  \\
&&\left. -\left( I-\alpha _{n}\mu F\right) T^{n}x^{\ast },x_{n+1}-x^{\ast
}\right\rangle  \\
&=&\alpha _{n}\rho \left\langle Vx_{n}-Vx^{\ast },x_{n+1}-x^{\ast
}\right\rangle +\alpha _{n}\left\langle \rho Vx^{\ast }-\mu Fx^{\ast
},x_{n+1}-x^{\ast }\right\rangle  \\
&&+\left\langle \left( I-\alpha _{n}\mu F\right) T^{n}y_{n}-\left( I-\alpha
_{n}\mu F\right) T^{n}x^{\ast },x_{n+1}-x^{\ast }\right\rangle  \\
&\leq &\alpha _{n}\rho \gamma \left\Vert x_{n}-x^{\ast }\right\Vert
\left\Vert x_{n+1}-x^{\ast }\right\Vert +\alpha _{n}\left\langle \rho
Vx^{\ast }-\mu Fx^{\ast },x_{n+1}-x^{\ast }\right\rangle  \\
&&+\left( 1-\alpha _{n}\nu \right) \left( \left\Vert y_{n}-x^{\ast
}\right\Vert +a_{n}\right) \left\Vert x_{n+1}-x^{\ast }\right\Vert .
\end{eqnarray*}%
Also, by using the inequality (\ref{2}\textbf{), }we get 
\begin{eqnarray*}
\left\Vert x_{n+1}-x^{\ast }\right\Vert ^{2} &\leq &\alpha _{n}\rho \gamma
\left\Vert x_{n}-x^{\ast }\right\Vert \left\Vert x_{n+1}-x^{\ast
}\right\Vert +\alpha _{n}\left\langle \rho Vx^{\ast }-\mu Fx^{\ast
},x_{n+1}-x^{\ast }\right\rangle  \\
&&+\left( 1-\alpha _{n}\nu \right) \left( \left\Vert x_{n}-x^{\ast
}\right\Vert +\beta _{n}\left\Vert Sx^{\ast }-x^{\ast }\right\Vert
+a_{n}\right) \left\Vert x_{n+1}-x^{\ast }\right\Vert  \\
&\leq &\left( 1-\alpha _{n}\left( \nu -\rho \gamma \right) \right)
\left\Vert x_{n}-x^{\ast }\right\Vert \left\Vert x_{n+1}-x^{\ast
}\right\Vert  \\
&&+\alpha _{n}\left\langle \rho Vx^{\ast }-\mu Fx^{\ast },x_{n+1}-x^{\ast
}\right\rangle  \\
&&+\left( 1-\alpha _{n}\nu \right) \beta _{n}\left\Vert Sx^{\ast }-x^{\ast
}\right\Vert \left\Vert x_{n+1}-x^{\ast }\right\Vert  \\
&&+\left( 1-\alpha _{n}\nu \right) a_{n}\left\Vert x_{n+1}-x^{\ast
}\right\Vert  \\
&\leq &\frac{\left( 1-\alpha _{n}\left( \nu -\rho \gamma \right) \right) }{2}%
\left( \left\Vert x_{n}-x^{\ast }\right\Vert ^{2}+\left\Vert x_{n+1}-x^{\ast
}\right\Vert ^{2}\right)  \\
&&+\alpha _{n}\left\langle \rho Vx^{\ast }-\mu Fx^{\ast },x_{n+1}-x^{\ast
}\right\rangle  \\
&&+\beta _{n}\left\Vert Sx^{\ast }-x^{\ast }\right\Vert \left\Vert
x_{n+1}-x^{\ast }\right\Vert +a_{n}\left\Vert x_{n+1}-x^{\ast }\right\Vert ,
\end{eqnarray*}%
which implies that%
\begin{eqnarray*}
\left\Vert x_{n+1}-x^{\ast }\right\Vert ^{2} &\leq &\frac{\left( 1-\alpha
_{n}\left( \nu -\rho \gamma \right) \right) }{\left( 1+\alpha _{n}\left( \nu
-\rho \gamma \right) \right) }\left\Vert x_{n}-x^{\ast }\right\Vert ^{2} \\
&&+\frac{2\alpha _{n}}{\left( 1+\alpha _{n}\left( \nu -\rho \gamma \right)
\right) }\left\langle \rho Vx^{\ast }-\mu Fx^{\ast },x_{n+1}-x^{\ast
}\right\rangle  \\
&&+\frac{2\beta _{n}}{\left( 1+\alpha _{n}\left( \nu -\rho \gamma \right)
\right) }\left\Vert Sx^{\ast }-x^{\ast }\right\Vert \left\Vert
x_{n+1}-x^{\ast }\right\Vert  \\
&&+\frac{2a_{n}}{\left( 1+\alpha _{n}\left( \nu -\rho \gamma \right) \right) 
}\left\Vert x_{n+1}-x^{\ast }\right\Vert  \\
&\leq &\left( 1-\alpha _{n}\left( \nu -\rho \gamma \right) \right)
\left\Vert x_{n}-x^{\ast }\right\Vert ^{2}+\alpha _{n}\left( \nu -\rho
\gamma \right) \theta _{n}\text{,}
\end{eqnarray*}%
where%
\begin{equation*}
\theta _{n}=\frac{2}{\left( 1+\alpha _{n}\left( \nu -\rho \gamma \right)
\right) \left( \nu -\rho \gamma \right) }\left[ 
\begin{array}{c}
\left\langle \rho Vx^{\ast }-\mu Fx^{\ast },x_{n+1}-x^{\ast }\right\rangle 
\\ 
+\frac{\beta _{n}}{\alpha _{n}}M_{3}+\frac{a_{n}}{\alpha _{n}}\left\Vert
x_{n+1}-x^{\ast }\right\Vert 
\end{array}%
\right] \text{,}
\end{equation*}%
and 
\begin{equation*}
\sup_{n\geq 1}\left\{ \left\Vert Sx^{\ast }-x^{\ast }\right\Vert \left\Vert
x_{n+1}-x^{\ast }\right\Vert \right\} \leq M_{3}.
\end{equation*}%
From condition (ii), by using Step 4, we get%
\begin{equation*}
\limsup_{n\rightarrow \infty }\theta _{n}\leq 0.
\end{equation*}%
So, it follows from Lemma \ref{Y} that the sequence $\left\{ x_{n}\right\} $
generated by (\ref{4}) converges strongly to $x^{\ast }\in Fix\left(
T\right) $ which is the unique solution of the variational inequality (\ref%
{v2}). This completes the proof.
\end{proof}

\begin{remark}
In particular, the point $x^{\ast }$\ is the minimum norm fixed point of $T,$
namely, $x^{\ast }$\ is the unique solution of the quadratic minimization
problem%
\begin{equation}
x^{\ast }=\func{argmin}_{x\in Fix\left( T\right) }\left\Vert x\right\Vert
^{2}\text{.}  \label{d3}
\end{equation}
Indeed, since the point $x^{\ast }$ is the unique solution of the
variational inequality (\ref{v2}), if we take $V=0$ and $F=I,$ then we get%
\begin{equation*}
\left\langle \left\langle \mu x^{\ast },x^{\ast }-x\right\rangle \leq 0,%
\text{ }\forall x\in Fix\left( T\right) \right\rangle .
\end{equation*}%
So we have%
\begin{equation*}
\left\langle x^{\ast },x^{\ast }-x\right\rangle =\left\langle x^{\ast
},x^{\ast }\right\rangle -\left\langle x^{\ast },x\right\rangle \leq
0\Longrightarrow \left\Vert x^{\ast }\right\Vert ^{2}\leq \left\Vert x^{\ast
}\right\Vert \left\Vert x\right\Vert .
\end{equation*}%
Hence, $x^{\ast }$ is the unique solution to the quadratic minimization
problem (\ref{d3}).
\end{remark}

Since a nearly nonexpansive mapping can be reduced to a nonexpansive mapping
by taking the sequence $\left\{ a_{n}\right\} $ as a zero sequence, under
the appropriate changings on the control sequences arising from Lemma \ref{Y}%
, we can derive main results of Wang and Xu \cite[Theorem 3.1]{wang} and
Ceng et. al. \cite[Theorem 3.1]{ceng} as following corollaries.

\begin{corollary}
Let $C$ be a nonempty closed convex subset of a real Hilbert space $H.$ Let $%
S,T:C\rightarrow C$ be nonexpansive mappings such that $Fix\left( T\right)
\neq \emptyset $. Let $V:C\rightarrow H$ be a $\gamma $- Lipschitzian
mapping, $F:C\rightarrow H$ be a $L$-Lipschitzian and $\eta $-strongly
monotone operator such that these coefficients satisfy $0<\mu <\frac{2\eta }{%
L^{2}}$, $0\leq \rho \gamma <\nu $, where $\nu =1-\sqrt{1-\mu \left( 2\eta
-\mu L^{2}\right) }$. For an arbitrarily initial value $x_{1}\in C,$
consider the sequence $\left\{ x_{n}\right\} $ in $C$ generated by (\ref{f4}%
) where $\left\{ \alpha _{n}\right\} $ and $\left\{ \beta _{n}\right\} $ are
sequences in $\left[ 0,1\right] $ satisfying the conditions:%
\begin{eqnarray*}
&(i)&\lim_{n\rightarrow \infty }\alpha _{n}=0\text{, and }%
\tsum_{n=1}^{\infty }\alpha _{n}=\infty \text{; \ \ \ \ \ \ \ \ \ \ \ \ \ \
\ \ \ \ \ \ \ \ \ \ \ \ \ \ \ \ \ \ \ \ \ \ \ \ \ \ \ \ \ \ \ \ \ \ \ \ \ }
\\
&(ii)&\lim_{n\rightarrow \infty }\frac{\beta _{n}}{\alpha _{n}}=0\text{, } \\
&(iii)&\tsum_{n=1}^{\infty }\left\vert \alpha _{n+1}-\alpha _{n}\right\vert
<\infty \text{ and }\tsum_{n=1}^{\infty }\left\vert \beta _{n+1}-\beta
_{n}\right\vert <\infty \text{. \ \ \ \ \ }
\end{eqnarray*}%
Then, the sequence $\left\{ x_{n}\right\} $ converges strongly to $x^{\ast
}\in Fix\left( T\right) $, where $x^{\ast }$\ is the unique solution of the
variational inequality (\ref{v2}). In particular, the point $x^{\ast }$\ is
the minimum norm fixed point of $T,$ that is $x^{\ast }$\ is the unique
solution of the quadratic minimization problem (\ref{d3}).
\end{corollary}

\begin{proof}
In the proof of Theorem \ref{X}, for all $n\geq 1$, if we take $T_{n}=T$
such that $T$ is a nonexpansive mapping , then the desired conclusion is
obtained.
\end{proof}

\begin{corollary}
Let $C$ be a nonempty closed convex subset of a real Hilbert space $H.$ Let $%
T:C\rightarrow C$ be a nonexpansive mapping such that $Fix\left( T\right)
\neq \emptyset $. Let $V:C\rightarrow H$ be a $\gamma $- Lipschitzian
mapping with, $F:C\rightarrow H$ be a $L$- Lipschitzian and $\eta $-
strongly monotone operator such that these coefficients satisfy $0<\mu <%
\frac{2\eta }{L^{2}}$, $0\leq \rho \gamma <\nu $, where $\nu =1-\sqrt{1-\mu
\left( 2\eta -\mu L^{2}\right) \text{. }}$ For an arbitrarily initial value $%
x_{1}\in C,$ consider the sequence $\left\{ x_{n}\right\} $ in $C$ generated
by (\ref{f2}) where $\left\{ \alpha _{n}\right\} $ and $\left\{ \beta
_{n}\right\} $ are sequences in $\left[ 0,1\right] $ satisfying the
conditions:$.$%
\begin{eqnarray*}
&\text{(i)}&\lim_{n\rightarrow \infty }\alpha _{n}=0\text{ and }%
\tsum_{n=1}^{\infty }\alpha _{n}=\infty \text{;} \\
&\text{(ii)}&\text{either }\tsum_{n=1}^{\infty }\left\vert \alpha
_{n+1}-\alpha _{n}\right\vert <\infty \text{ or }\lim_{n\rightarrow \infty }%
\frac{\alpha _{n+1}}{\alpha _{n}}=0,\text{\ \ \ \ \ \ \ \ \ \ \ \ \ \ \ \ \
\ \ \ \ \ \ \ \ \ \ \ \ \ \ \ \ \ \ \ }\text{\ }\text{\ }
\end{eqnarray*}%
Then, the sequence $\left\{ x_{n}\right\} $ converges strongly to $x^{\ast
}\in Fix\left( T\right) $, where $x^{\ast }$\ is the unique solution of the
variational inequality (\ref{v2}).
\end{corollary}

\begin{proof}
In the proof of Theorem \ref{X}, let $S=I$ where $I$ is the identity mapping
and $T$ be a nonexpansive mapping. Then the proof is clear.
\end{proof}

Let $\left\{ T_{n}\right\} $ be a sequence of mappings from $C$ into $H$ and 
$\left\{ a_{m}^{n}\right\} $ be a sequence in $\left[ 0,\infty \right) $\
with $\lim_{m\rightarrow \infty }a_{m}^{n}=0$, for each $n\geq 1$. Then, $%
\left\{ T_{n}\right\} $ is called a sequence of nearly nonexpansive mappings
with respect to a sequence $\left\{ a_{m}^{n}\right\} $ if%
\begin{equation}
\left\Vert T_{n}^{m}x-T_{n}^{m}y\right\Vert \leq \left\Vert x-y\right\Vert
+a_{m}^{n}\text{, }\forall x,y\in C\text{, }n,m\geq 1.  \label{nearly1}
\end{equation}%
Considering Remark \ref{AA} and Remark \ref{AB}, it is to easy to see that
the sequence $\left\{ T_{n}\right\} $ defined by (\ref{nearly1}) is
different from the sequence defined by (\ref{nearly}) of Wong et. al. \cite%
{wong}. The iterative scheme defined by (\ref{4}) can be modified for the
sequence of nearly nonexpansive mappings defined by (\ref{nearly1}).
Accordingly, the problem written in the following remark arises.

\begin{remark}
For a sequence of nearly nonexpansive mappings (\ref{nearly1}) with a
nonempty common fixed points set, it is an open problem whether or not an
iteration process generated by this sequence\ will converge strongly to a
common fixed point.
\end{remark}

\end{document}